\documentclass[11pt,reqno]{amsart}

\usepackage[tmargin=1.5in,bmargin=1.4in,rmargin=1.4in,lmargin=1.4in]{geometry}

\usepackage{amsmath,amsthm,amssymb,mathrsfs}

\theoremstyle{plain}
\newtheorem{theorem}{Theorem}[section]
\newtheorem{corollary}[theorem]{Corollary}

\newtheorem{proposition}[theorem]{Proposition}

\theoremstyle{definition}
\newtheorem{definition}[theorem]{Definition}

\numberwithin{equation}{section}
\numberwithin{table}{section}

\newcommand{\C}{\mathbb{C}}

\newcommand{\D}{\mathbb{D}}
\newcommand{\R}{\mathbb{R}}
\newcommand{\T}{\mathbb{T}}

\begin{document}
\title{Moment estimates of the cloud of a planar measure}

\author{Mihai Putinar}
\address[M.~Putinar]{University of California at Santa Barbara, CA,
USA and Newcastle University, Newcastle upon Tyne, UK} 
\email{\tt mputinar@math.ucsb.edu, mihai.putinar@ncl.ac.uk}

\date{\today}



\begin{abstract} With a proper function theoretic definition of the {\it cloud} of a positive measure with compact support in the real plane, 
a computational scheme of transforming the moments of the original measure into the moments of the uniformly distributed mass on the cloud
is described. The main limiting operation involves exclusively  truncated Christoffel-Darboux kernels, while
error bounds depend on the spectral asymptotics of a Hankel kernel belonging to the Hilbert-Schmidt class.

\end{abstract}

\maketitle

\section{Introduction} This is a continuation of a recent work \cite{Putinar-2019}, adding to it some constructive analysis details. The purpose of both articles was/is to 
compute the moments of the ``cloud'' of a positive measure with compact support in the plane starting from the moments of the original measure.
This moment conversion eliminates the components of the measure which are singular with respect to Lebesgue area, informally regarded as ``outliers",
and charges the cloud with uniform mass. Two real dimensions are specific in the approximation scheme we propose, twofold: complex variables enter heavily into the picture,
as well as a refined spectral analysis for pairs of self-adjoint operators with trace-class commutator. 

The present article has two interlaced motivations: to identify a 2D analogue of the 
classical and still rapidly evolving spectral analysis on the line or the circle, based on moment information, involving orthogonal polynomials and Christofell-Darboux kernels, and second,
the current renewed interest of data analysis flavor, consisting in separating outliers from the cloud, in multidimensional point distributions. Our viewpoint is to treat moments as observed data and to extract from them by simple analytic transforms the moments of the cloud. We only mention without expanding in this work the well known, various reconstruction from moments methods.

A landmark contribution to polynomial approximation theory in the complex domain is Thomson's Theorem \cite{Thomson-1991}. It asserts that complex polynomials are either
dense in the Lebesgue space $L^2(\mu)$ associated to a positive Borel measure $\mu$ with compact support in the plane, or there exist $L^2(\mu)$-bounded point evaluations for polynomials.
More specifically, Thomson's Theorem provides a decomposition of the measure $\mu$ into a singular with respect to area part, where polynomials are dense, and measures which carry an open set
of point evaluations, see Theorem \ref{Thomson} below for the precise statement. The multiplier with the complex variable $S_\mu = M_z$, acting on the closure of polynomials $P^2(\mu)$, is consequently decomposed into a direct sum of a normal operator and a collections of irreducible subnormal, cyclic operators with a continuum of eigenvalues for their adjoints: $S_\mu = N \oplus S_1 \oplus S_2 \oplus \ldots. $ In very elementary terms, the main idea we exploit in this article is: the self-commutator 
$$[S_\mu^\ast, S_\mu] = 0 \oplus [S_1^\ast, S_1] \oplus [S_2^\ast, S_2] \oplus \ldots$$ does not ``see" the normal part, that is the outliers to the cloud.

The spectral analysis of the analytic Toeplitz operator $S_\mu$ is rich and provides the numerical procedure we propose to eliminate via traces of commutators its normal direct summand. 
A second notable result we invoke is
due to Berger and Shaw \cite{Berger-Shaw-1974}, namely the self-commutator $[S^\ast_\mu, S_\mu]$ is trace class. This opens a wide array of tools specific to the notion of principal function,
a refined spectral invariant extending the Fredholm index across the spectrum. A trace formula discovered in the 1970-ies by Helton and Howe and an equivalent determinant formula 
independently obtained by Carey and Pincus provide effective formulas linking the moments of the principal function (in our case the characteristic function of the cloud of the measure $\mu$) and traces of commutators of smooth functions applied to $S_\mu.$
The third deep result we rely on is due to Carey and Pincus, asserting that the principal function of a subnormal operator is integer valued \cite{Carey-Pincus-1981}.

This machinery was put in action in our previous article \cite{Putinar-2019} with a proposed computational/transformation scheme:
\bigskip

{\it moments of $\mu$ $ \longrightarrow$ \  moments of the uniform mass on the cloud of $\mu$.}
\bigskip

We expand below a quantitative analysis of the algorithm, with error bounds expressed in the entries of the Hessenberg matrix representing $S_\mu$ with respect to the filtration
of $P^2(\mu)$ by finite dimensional subspaces consisting of polynomials. The main transform involves solely limits of some explicit expressions involving truncated Christoffel-Darboux kernels of 
the original measure $\mu$. And not surprising, the principal error estimate depends on the asymptotics of the singular numbers of Hankel's operator $(I-P)M_z^\ast P$, where $P : L^2(\mu) \longrightarrow P^2(\mu)$ denotes the orthogonal projection. The approximation scheme can be adapted to other orthogonal function systems than polynomials, with an appropriate sequence of finite rank projections converging strongly to the identity operator. We remain at an all-inclusive level, applicable to {\it any} positive measure with compact support, leaving for future studies the adaptation of this general framework to concrete situations. The examples at the end of the article indicate a few openings in this direction.

To make the present note accessible to a larger group of readers, we (re)expose below the main ingredients, with precise reference to the sources. 
\bigskip

{\bf Acknowledgements.} We are grateful to Bernhard Beckermann for helpful comments on the asymptotics of the Hessenberg matrix associated to a uniform mass distribution
in the plane.

\section{Preliminaries} Throughout this note $\C[z]$ stands for the algebra of polynomials in one variable with complex coefficients.
Let $\mu$ be a positive Borel measure with compact support in the complex plane. We assume the support of $\mu$
is not finite. 

\subsection{Orthogonal polynomials} We recall a few facts and conventions referring to complex orthogonal polynomials, as for instance appearing in \cite{Stahl-Totik-1992}.
The closure of complex polynomials in $L^2(\mu)$ is denoted $P^2(\mu)$, with associated orthogonal projection
$P$. The multiplication $M$ by the complex variable $z \in \C$ is a bounded linear transform of $L^2(\mu)$
which leaves invariant the subspace $P^2(\mu)$:
$$ P M P = M P.$$
The linear operator $S_\mu = M|_{P^2(\mu)}: P^2(\mu) \longrightarrow P^2(\mu)$ is called {\it subnormal}. It is also a Toeplitz operator:
$$ (S_\mu f)(z) = P (w f(w))(z), \ \ f \in P^2(\mu).$$
Sometimes we simply write $S = S_\mu$.
The spectrum of $M$ coincides with the closed support ${\rm supp}(\mu)$ of the measure $\mu$, while the spectrum of $S$ can be larger,
containing in addition some connected components of $\C \setminus {\rm supp}(\mu)$.

The constant function $\mathbf{1}$ is a cyclic vector for $S$, producing the finite dimensional filtration (Krylov subspaces):
$$ \C_n[z] = \{ f \in \C[z],\  \deg f \leq n \} = {\rm span} \{ S^j \mathbf{1}, \ 0 \leq j \leq n \}.$$
We denote by $p_n(z)$ the associated complex orthogonal polynomials:
$$ \langle p_j, p_k \rangle = \int p_j \overline{p_k} d\mu = \delta_{jk}, \ \ \deg p_k =k, \ \ j,k \geq 0.$$
The operator $S$ has a distinguished Hessenberg matrix representation with respect to the orthonormal basis
$(p_n)_{n=0}^\infty$ of $P^2(\mu)$:
$$ \langle z p_j, p_k \rangle = h_{kj}, \ j,k \geq 0,$$
observing the automatic vanishing relations
$$ h_{jk} =0, \ j+1 < k.$$
The adjoint operator is represented by the matrix:
$$ \langle S^\ast p_j, p_k \rangle = \langle M^\ast p_j, p_k \rangle = \langle p_j , S p_k \rangle = \overline{h_{jk}}.$$
It is customary to normalize the leading coefficient of $p_n$ to be a positive number:
$$ p_n(z) = \gamma_n z^n + \ldots, \ \ \gamma_n > 0, \ \ n \geq 0.$$ The matrix representing $S$ has the form:
$$
H = \begin{bmatrix}
h_{00} & h_{01} & h_{02} & h_{03}& \ldots\\
h_{10}&h_{11}& h_{12} & h_{13} & \ldots\\
0& h_{21}&h_{22} &h_{23}& \\
0&0&h_{32}&h_{33}& \ldots\\
\vdots & & \ddots & \ddots 
\end{bmatrix}.
$$
If the measure $\mu$ has support on the real line, then $S$ is a self-adjoint operator, hence the matrix $H$ is symmetric, with real values, with
only three non-vanishing diagonals: a well charted Jacobi matrix framework.

For every non-negative integer $n$ one can speak of the reproducing kernel (known as the {\it Christoffel-Darboux kernel}) of the finite dimensional subspace $\C_n[z]$:
$$ K_n(z,w) = \sum_{j=0}^n p_j (z) \overline{p_j(w)},$$
characterized by the identity:
$$ \langle K_n(z,w), f(w) \rangle = 
\left\{ \begin{array}{cr}f(z), &f \in \C_n[z],\\
                       0, & {\rm deg} f >n. \end{array} \right. $$
The {\it Christoffel function} of order $n$ is
$$ \Lambda_n (z) = \frac{1}{K_n(z,z)} = \inf \{ \| f \|^2, \ f(z) =1, \ \deg f \leq n \}.$$
These numerical values decrease with $n$, and the limit 
$$ \Lambda^\mu (z) = \Lambda (z) = \inf_n  \Lambda_n (z) $$
is known as the Christoffel function associated to the measure $\mu$, evaluated at the point $z$.

Incidentally, the leading coefficient of the orthogonal polynomial $p_n$ has a similar variational interpretation:
$$ \frac{1}{\gamma_n} = \inf \{ \| z^n - f(z) \|, \ \deg f \leq n-1 \}, \ n \geq 1.$$

\subsection{Trace and determinant formulas} 

The analytic subnormal operator $S  = PMP$ satisfies the {\it hyponormal} commutator inequality
$ [S^\ast, S] \geq 0.$ Indeed, for any element $f \in P^2(\mu)$:
$$ \langle [S^\ast, S]f, f \rangle = \| S f\|^2 - \| S^\ast f\|^2 = $$ $$
\| w f(w) \|^2 - \| P \overline{w} f(w) \|^2 \geq  \| w f(w) \|^2 - \| \overline{w}f(w) \|^2 = 0.$$
A theorem of Berger and Shaw \cite{Berger-Shaw-1974} asserts that the above commutator is trace-class: ${\rm Tr} [S^\ast,S] < \infty$. A different proof of Berger-Shaw Theorem is due to Voiculescu \cite{Voiculescu-1980'}, in an article putting this very trace bound in the context of general perturbation theory, with a specific link to the modulus of quasi-triangularity of a Hilbert space linear transform.

For any polynomials $p \in \C[z,\overline{z}]$, one can define an ordered functional calculus (traditionally called {\it hereditary functional calculus}) $p(S, S^\ast)$, by arranging all powers of $S^\ast$ to the left of the powers of $S$, in every monomial. Then for a pair of polynomials $p,q \in \C[z,\overline{z}]$ one proves by degree induction that the commutator $[p(S,S^\ast), q(S,S^\ast)]$ is trace-class. A remarkable 
observation due to Helton and Howe \cite{Helton-Howe-1975} asserts that the bilinear form
$$ (p,q) \mapsto {\rm Tr}[p(S,S^\ast), q(S,S^\ast)] $$
depends linearly, and it is continuous in the sense of distributions, on the Jacobian  
$$J(p,q) = \frac{\partial p}{\partial \overline{z}} \frac{\partial q}{\partial {z}} - \frac{\partial q}{\partial \overline{z}} \frac{\partial p}{\partial {z}}.$$
Further on, it was established by Carey and Pincus as a byproduct of a decade of groundbreaking discoveries \cite{Carey-Pincus-1981}  that there exists a Borel measurable set $\Sigma(\mu)$ satisfying
\begin{equation}\label{HH}
 {\rm Tr}[p(S,S^\ast), q(S,S^\ast)]  = \int_{\Sigma(\mu)} J(p,q) dA, \ \ p, q \in \C[z,\overline{z}].
\end{equation}
Above, $dA$ stands for Lebesgue measure in $\R^2$. 

\begin{definition} The {\it cloud} of a positive Boreal measure $\mu$ with compact support in $\C$ is the measurable set $\Sigma(\mu)$ appearing in trace formula (\ref{HH}).
\end{definition}

Note that the set $\Sigma(\mu)$ is only determined up to area null-sets. In practice it is more appropriate to speak about
the class $[\chi_{\Sigma(\mu)}]$ of its characteristic function in $L^1(\C,\mu)$.

The equivalent formulation as an infinite determinant formula goes back to the origins of Carey and Pincus work:
\begin{equation}\label{det}
\det [ (S-w) (S^\ast-\overline{z}) (S-w)^{-1}(S^\ast-\overline{z})^{-1}] = $$  $$
\exp ( \frac{-1}{\pi} \int_{\Sigma(\mu)} \frac{\rm dA(\zeta)}{(\zeta-w)(\overline{\zeta}-\overline{z})}), \ \ |z|, |w| > \| S \|.
\end{equation}

All  results above touch the surface of the theory of hyponormal operators, in particular with reference to the {\it principal function} of a semi-normal operator. We do not expand here the details, referring to the monograph \cite{Martin-Putinar-1989} for complete proofs and historical comments. In the present article we focus on the reconstruction of the cloud
$\Sigma(\mu)$ contained in the polynomial convex hull of the support of the measure $\mu$. Note for instance that
$$ {\rm Tr}[S^\ast, S] = \frac{1}{\pi} {\rm Area}\ \Sigma(\mu),$$
hence the operator $S$ is normal if and only if $ {\rm Area} \ \Sigma(\mu) =0$. We will see shortly that this in turn is equivalent to the density of complex polynomials in $L^2(\mu)$, quite a central topics in approximation theory.

For the limited aim of the present work, we confined ourselves to reproduce some well-known computations specific to Hankel operators.
The general theory of Hankel operators is superbly exposed in \cite{Peller-2003}. The following observations appeared in a slightly different form in \cite{Carey-Pincus-1979}, but they are not uncommon
to the general spectral analysis of Toeplitz operators, see for instance \cite{Zhu-2019}.

More precisely, the self-commutator of the subnormal operator $S = PMP = MP$ can be factored as:
$$ [S^\ast, S] = [PM^\ast P, PMP] = P M^\ast P M P - PMP M^\ast P = $$ $$P M^\ast M P - PMP M^\ast P = P M M^\ast PMP M^\ast P = $$ $$ P M (I-P) M^\ast P =
[P, M] [M^\ast, P] = [P, M] [P, M]^\ast.$$
Note that $MP = PMP$ and $PM^\ast = PM^\ast P$.
Therefore the linear operator (known as a {\it Hankel operator})
$$ T = (I-P)M^\ast P : P^2(\mu) \longrightarrow L^2(\mu) \ominus P^2(\mu)$$ 
is Hilbert-Schmidt: ${\rm Tr} \ T^\ast T < \infty$. We can express this fact in two different ways:
\begin{equation}\label{HS} 
\sum_{k=0}^\infty \| T p_k \|^2 = \sum_{k=0}^\infty \| \overline{z} p_k(z) - P [\overline{w} p_k(w)](z) \|^2 < \infty,
\end{equation}
or, there exists a kernel function $L \in L^2(\C \times \C, \mu \otimes \mu)$ satisfying
$$ [T f ](z) = \int L(z, w) f(w) d\mu(w), \ \ f \in P^2(\mu),$$
and in particular
$$ {\rm Tr}  \ T^\ast T = \int |L(z,w)|^2 d \mu(z) d \mu(w).$$
The singular numbers of $T$ will play a central role below.

\subsection{Function theory}

The fine structure of the subnormal operator $S = S_\mu$ was elucidated by a rather recent discovery by Jim Thomson. We state in full the main theorem.

\begin{theorem}[Thomson]\label{Thomson} Let $\mu$ be a positive Borel measure, compactly supported on $\C$. There exists a Borel partition $\Delta_0, \Delta_1, \ldots $
of the closed support of $\mu$ with the following properties:
\bigskip 

1) $P^2(\mu) = L^2(\mu_0) \oplus P^2(\mu_1) \oplus P^2(\mu_2) \oplus \ldots$, where $\mu_j = \mu|_{\Delta_j}, \ j \geq 0;$
\bigskip

2) Every operator $S_{\mu_j}, \ j \geq 1,$ is irreducible with spectral picture:
$$ \sigma(S_{\mu_j}) \setminus \sigma_{\rm ess}(S_{\mu_j}) = G_j,\ \ {\it simply \ connected},$$
and
$$ {\rm supp} \mu_j \subset \overline{G_j}, \ \ j \geq 1;$$
\bigskip

3) If $\mu_0 = 0$, then any element $f \in P^2(\mu)$ which vanishes $[\mu]$-a.e. on $G = \cup_j G_j$ is identically zero.

\end{theorem}

The proof appeared in \cite{Thomson-1991}, for $L^p$ spaces, $1 \leq p < \infty.$  A conceptually simpler proof appears in \cite{Brennan-2005}.
The central position of Thomson's Theorem was immediately recognized in the monograph \cite{Conway-1991} (published almost simultaneously with the original article). 

\begin{corollary} The cloud $\Sigma(\mu)$ of a positive measure $\mu$ with compact support in $\C$ is empty if and only if the complex polynomials are dense in $L^2(\mu)$.
In case $\Sigma(\mu)$ is non-empty, it contains interior points: $G_j \subset \Sigma(\mu), \ \ j \geq 1.$
\end{corollary}

The operator $S_{\mu_0} = M_z \in {\mathcal L}(L^2(\mu_0))$ is the single normal component of $S_\mu$ while the summands $S_{\mu_j}$ collect the non-normal behavior of
$S_\mu$. For every index $j, j \geq 1,$
the theorem asserts that $S_{\mu_j}$ admits a continuum of eigenvalues of multiplicity one, filling the simply connected open set $G_j$:
$$ \lambda \in G_j \ \ \Rightarrow [\ker (S_{\mu_j} -\lambda) = 0, \ \ \dim \ker(S_{\mu_j}^\ast - \overline{\lambda}) = 1].$$
Moreover, it is known that the range of $S_{\mu_j}-\lambda$ is closed for such spectral parameters. The corresponding eigenvectors span $P^2(\mu_j)$, to the extent that this functional Hilbert space carries a
reproducing kernel. This feature is detected by the local Christoffel function:
$$ \Lambda^{\mu_j}(\lambda) : = \inf \{ \|f \|^2_{\mu_j}; f \in \C[z], \ f(\lambda) =1\} >0,$$
and ultimately by the full Christoffel function:
\begin{equation}
\Lambda (\lambda) = \Lambda^\mu (\lambda)  = \inf \{ \|f \|^2_{\mu}; f \in \C[z], \ f(\lambda) =1\} >0, \ \lambda \in G_j, \ j \geq 1.
\end{equation}

\begin{corollary} The cloud of a positive measure with compact support $\mu$ is non-empty if and only if
there exists a point $\lambda \in \C$ satisfying $\Lambda^\mu(\lambda) > 0$.
\end{corollary}

Szeg\"o's theory of orthogonal polynomials on the unit circle $\T$ provides a clear cut picture: if ${\rm supp}(\mu) \subset \T$, then $\Sigma(\mu)$ is either empty or equal to the full disk $\overline{\D}$.
The latter case is characterized by Szeg\"o's condition:
$$ \int_\T  | \log \frac{d \mu}{d \theta}| d \theta < \infty,$$
or equivalently by Christoffel function test: for at least one point $\lambda \in \D$ one has
$$ \Lambda^\mu(\lambda) >0,$$
and then the same is true for all $z \in \D$. See for instance \cite{Ahiezer-1965}.

Another, related and notable particular case is represented by measures $\mu$ subject to the finiteness condition
$$ {\rm rank} [S^\ast, S] < \infty.$$ This situation was fully analyzed by McCarthy and Yang \cite{McCarthy-Yang-1995,McCarthy-Yang-1997} with the following conclusion:
Thomson's decomposition admits finitely many irreducible summands (that is $0 \leq j \leq n$) with the normal part represented by a positive measure $\mu_0$ which is
singular with respect to harmonic measure on every $G_j, 1 \leq j \leq n,$ and does not put mass on $\cup_{j=1}^n G_j$; in addition every simply connected component $G_j$ is a quadrature domain.
In their theorem a complete description of the measures $\mu_j, 1 \leq j \leq n,$ is provided. Moreover, under this finite rank condition, the intersection of any two distinct sets
$\overline{G_j} \cap \overline{G_k}$ contains at most one single point, and the union $F = \cup_{j=1}^n \overline{G_j}$ is polynomially convex. In other terms the complement
$\C \setminus F$ is connected.

A bounded open set $\Omega$ of the complex plane
is called a {\it quadrature domain} for analytic functions if there is a distribution of finite support $\tau$ in $\Omega$ (combination of point masses and their derivatives), such that
$$ \int_\Omega f {\rm d A} = \tau(f), \ \ f \in L^1_a(\Omega, {\rm d A}),$$
where $L^1_a(\Omega, {\rm d A})$ stands for the space of analytic functions in $\Omega$ which are Lebesgue integrable. 
The {\it order} of a quadrature domain is the number of nodes, counting multiplicity, in the above cubature formula. A connected quadrature domain has an irreducible real algebraic boundary. The simplest example is of course a disk. Indeed, on the unit disk
$$ \int_\D f(z) dA(z) = \pi f(0), \ \ f \in  L^1_a(\D, {\rm d A})$$
by Gauss mean value theorem.
Any simply connected quadrature domain is a conformal image of the disk by a rational function, for instance a cardiodid or a connected lemniscate of degree four. An informative and accessible survey of quadrature domains, with ramifications to potential theory, inverse geophysical problems, quantum physics and operator theory is \cite{Gustafsson-Shapiro-2005}.

We record these facts under a precise statement.

\begin{proposition} Let $\mu$ be a positive Borel measure with compact support in $\C$. The cloud $\Sigma(\mu)$ appearing in trace formula (\ref{HH})
contains every open set $G_j, \ j \neq 0,$ appearing in Thomson's decomposition of the operator $S_\mu$.

In case ${\rm rank} [S^\ast_\mu, S_\mu] < \infty$ there are only finitely many summands ( $ 1 \leq j \leq n$) and the cloud $\Sigma(\mu)$ coincides up to an area null-set with the quadrature domain $\Omega = G_1 \cup \ldots \cup G_n.$
\end{proposition} 

The nature of the measure $\mu$ is also quite explicit in the case of finite rank self-commutator, as follows. Let $r : \D \longrightarrow \Omega$ denote a conformal rational map onto a bounded quadrature domain $\Omega$. Let $\tau$ be an absolutely continuous measure with respect to
harmonic measure $\omega$ for $\Omega$, supported on $\partial \Omega$ and satisfying
$$ \int_{\partial \Omega} |\log (\frac{ d\tau}{d\omega})| d\omega < \infty.$$
Let $\mu = \tau + \nu$, where $\nu$ is a positive, finitely supported measure on $\Omega$. Then
complex polynomials are not dense in $L^2(\mu)$ and the operator $S_\mu$ is a typical irreducible subnormal operator possessing finite rank self-commutator, cf. Theorem 1.12 in \cite{McCarthy-Yang-1997}.
In this case $\overline{\Omega}$ coincides with the cloud $\Sigma(\mu)$.

\section{Finite rank approximation}

\subsection{Commutator inequalities}

Let $\phi \in {\mathcal C}^\infty (\C)$. The operator of multiplication by $\phi$ on $L^2(\mu)$ is denoted $M_\phi$ or by the same letter if there is no confusion. Obviously $M_\phi$ is bounded. Jacobi's identity:
$$ [M^\ast, [\phi, P]] + [\phi, [P,M^\ast]] + [\phi, [P,M^\ast]] =0$$
yields
$$ [M^\ast, [\phi, P]] = [\phi, [M^\ast, P]].$$
In other terms, for an element $f \in L^2[\mu]$ we find:
$$ \overline{z} ([\phi, P] f(w) ) (z) - ([\phi, P] \overline{w} f(w))(z) = $$ $$
\phi(z) ([M^\ast, P] f(w) ) (z) - ([M^\ast, P] \phi(w) f(w))(z).$$
From here we deduce at the level of integral kernels:
\begin{equation}\label{commutator}
([\phi, P] f)(z) = \int L(z,w) \frac{ \phi(z)-\phi(w)}{\overline{z} - \overline{w}} f(w) d \mu(w), \ \ f \in L^2(\mu).
\end{equation} 

In particular, the commutator $[\phi, P]$ is Hilbert-Schmidt for every smooth function $\phi$.
Observe that for $\phi$ complex analytic and $f \in P^2(\mu)$ one has $[\phi, P] f = \phi P f - P (\phi f) = \phi f - \phi f = 0.$

One step further, we can write explicitly the integral kernel of the commutator appearing in trace formula (\ref{HH}). For a polynomial $p \in \C[z,\overline{z}]$
either one computes $p(S,S^\ast)$ via an ordered functional calculus, or one takes $P M_p P$, the result differs by a trace class operator:
$$ p(S,S^\ast) - P M_p P \in {\mathcal C}_1.$$
Hence the trace of $[p(S,S^\ast), S]$ is not affected by such a rearrangement of terms. Pure algebra yields:
$$ [P M_p P, S] = P M_p M P - P M P M_p P = $$ $$ P (M M_p - M P M_p) P = P M (I-P) M_p P = [P, M] [ p, P].$$

Summing up,
\begin{equation}\label{integral}
{\rm Tr} [p(S,S^\ast), S] = \int \int | L(z,w)|^2 \frac{ p(z,\overline{z})-p(w,\overline{w})}{\overline{z} - \overline{w}} d \mu(z) d\mu(w), \ \ p \in \C[z, \overline{z}].
\end{equation}

The above formula extends to all Toeplitz operators $T_\phi = P M_\phi P$ with a smooth symbol. In addition, trace formula yields: 
$$ {\rm Tr} [T_\phi, S] = \frac{1}{\pi} \int_{\Sigma(\mu)} \frac{\partial \phi}{\partial \overline{z}} dA = $$ $$ \int \int | L(z,w)|^2 \frac{ \phi(z)-\phi(w)}{\overline{z} - \overline{w}} d \mu(z) d\mu(w), \ \ \phi \in {\mathcal C}^\infty (\C).$$

One can deduce from above a few qualitative properties of the integral kernel $L(z,w)$, by exploiting the distributional sense of the above identity. In the present article we are concerned only with some approximation estimates. The ad-hoc notation 
$$ (\Delta \phi)(z,w) =  \frac{ \phi(z)-\phi(w)}{\overline{z} - \overline{w}} $$
is adopted in the following pages, as well as the Besov type semi-norm
$$ | \phi |_F = \inf_{p \in \C[z]} \| \Delta ( \phi - p) \|_{\infty, F \times F},$$
where $F$ is a closed subset of the complex plane. Identity \ref{commutator} implies, for every smooth function $\phi$:
$$ \| [\phi, P] \|_{HS} \leq \| L \|_{2, \mu \times \mu} | \phi |_{{\rm supp} (\mu)}.$$

Let $Q$ denote an orthogonal projection of $L^2(\mu)$ onto a closed subspace of $P^2(\mu)$. For a fixed smooth function $\phi$ one finds, along the above computations: 
$$ {\rm Tr}\ Q [T_\phi, S] Q = {\rm Tr}\ Q [S, P] [T_\phi, P] = {\rm Tr}\ Q [S, P] [\phi, P] = {\rm Tr} \ Q [M, P] [\phi, P].$$
The truncated operator $[M, P]^\ast Q$ has finite rank and it is represented by an integral kernel:
$$ ([M, P]^\ast Q f) (z) = \int L_Q(z,w) f(w) d \mu(w), \ \ f \in P^2(\mu).$$
More importantly,
$$ |{\rm Tr} \ Q [T_\phi, S] Q - {\rm Tr} \  [T_\phi, S] | = | {\rm Tr} (I-Q)[M,P] [\phi, P] | \leq 
$$ $$ \|  (I-Q)[M,P] \|_{HS} \| [\phi, P] \|_{HS} \leq  \|  (I-Q)[M,P] \|_{HS} \| [M^\ast, P] \|_{HS} | \phi|_{{\rm supp} (\mu)}.$$
Recall also that
$$ \| [M^\ast, P] \|^2_{HS} = {\rm Tr} [S^\ast, S] =  \frac{ {\rm Area} \Sigma(\mu)}{\pi}.$$

In conclusion we have proved the following result.

\begin{theorem}\label{gap}
 Let $\mu$ be a positive Borel measure with compact support in $\C$ and let $\phi \in {\mathcal C}^\infty (\C)$. For an orthogonal projection $Q$ of
$L^2(\mu)$ onto a closed subspace of $P^2(\mu)$ the estimate
\begin{equation}\label{estimate}
|{\rm Tr} \ Q [T_\phi, S] Q - {\rm Tr} \  [T_\phi, S] | \leq  \sqrt{\frac{ {\rm Area} \Sigma(\mu)}{\pi}} \| (I-Q)[M,P] \|_{HS}  \ |\phi |_{{\rm supp} (\mu)}.
\end{equation}
holds true.
\end{theorem} 

Remark that one can express the distortion factor above in function space norm:
$$ \| (I-Q) [M,P] \|_{HS} = \| L  - L_Q \|_{2, \mu \otimes \mu}.$$

The particular case of a finite rank projection $Q$ onto a subspace of polynomials will be of interest in the next sections.

Hankel operator $T = [M^\ast, P] = (I-P) M^\ast P: P^2(\mu) \longrightarrow L^2(\mu) \ominus P^2(\mu)$ is Hilbert-Schmidt. Its Schmidt expansion
\begin{equation}\label{schmidt}
 (I-P)M^\ast P= \sum_{j=0}^\infty \kappa_j g_j \langle \cdot, f_j \rangle 
 \end{equation}
identifies the singular numbers $(\kappa_j)$ together with the eigenvectors of its modulus:
$$ [S^\ast, S] f_j = T^\ast T f_j = \kappa_j^2 f_j, \ j \geq 1.$$
By convention we include here the null vectors of $T$, so that $(f_j)_{j=0}^\infty$ is an orthonormal basis of $P^2(\mu)$. And similarly $g_j$ are eigenvectors of
$T T^\ast$, forming an orthonormal basis of $P^2(\mu)^\perp.$ Courant-Fisher's min-max principle implies the following bound.

\begin{corollary} In the conditions of the Theorem, let $\kappa_1 \geq \kappa_2 \geq \ldots \geq 0$ denote the singular numbers of Hankel's operator $[M^\ast, P] = (I-P)M^\ast P$.
If $Q$ is a  projection of finite rank $d$ onto a subspace of $P^2(\mu)$, then
$$ \| [M^\ast, P] (I-Q) \|^2_{HS} \geq  \sum_{j \geq d} \kappa_j^2,$$
and the inequality is attained if $Q$ is the projection onto the span of vectors \\
$f_0, f_1, \ldots, f_{d-1}$.
\end{corollary}

Although it might be inaccessible numerically to have the complete Schmidt expansion of the operator $T$, we pretend for a moment that expansion (\ref{schmidt}) is known.
Let $p \in \C[z,\overline{z}]$, so that
$$ [P M_p P, S] =  P M (I-P) M_p P$$
as before. Then
$$ \langle  P M (I-P) M_p P f_j, f_j \rangle = \langle p f_j , T f_j \rangle = \kappa_j \langle p f_j , g_j \rangle $$
hence
$$ {\rm Tr}  [P M_p P, S] = \lim_n \sum_{j=0}^n \kappa_j \int p(z,\overline{z}) f_j(z) \overline{g_j(z)} d\mu(z).$$
This observation is effective in case ${\rm rank} [S_\mu^\ast, S_\mu] < \infty$ when
$$  \sum_{j=0}^\infty \kappa_j f_j(z) \overline{g_j(z)} = \sum_{j=0}^d  \kappa_j f_j(z) \overline{g_j(z)} = \overline{ L(z,z)} $$
is an integrable function with respect to $\mu$.

These computations lead to a theoretical result, illustrating the key role played by the asymptotics of the singular numbers of the operator $T$.

\begin{theorem} Let $\mu$ be a positive Borel measure with compact support in $\C$ and let $p \in \C[z,\overline{z}]$. Assume that the singular numbers 
of Hankel's operator $[M^\ast,P]$ are in $\ell^1$. Then $L(z,z) \in L^1(\mu)$ and 
\begin{equation}
\frac{1}{\pi}  \int_{\Sigma(\mu)} \frac{\partial p}{\partial \overline{z}} (z,\overline{z}) dA(z) = \int p(z) \overline{L(z,z)} d\mu(z).
\end{equation}
\end{theorem}

Regardless to say that there are ample studies and criteria assuring a Hankel operator to be trace-class, see \cite{Peller-2003}.

\subsection{Weak approximation of the integral kernel}

In general, the integral kernel $L(z,w)$ representing the Hilbert-Schmidt operator $(I-P)M^\ast P : P^2(\mu) \longrightarrow L^2(\mu) \ominus P^2(\mu)$ is only an element of $L^2(\mu \times \mu)$.
Its structure is illuminated by an approximation in the weak topology of $L^2(\mu \times \mu)$ which involves truncated Christoffel-Darboux kernels.

To this aim, fix an element $f \in P^2(\mu)$ and remark that
$$ (I-P) M^\ast P f = \lim_n (I-P_n) M^\ast P f = \lim_n M^\ast P_{n+1} f - \lim_n P_n M^\ast P.
$$ But $P_n M^\ast = P_n M^\ast P_{n+1}$, hence the strong operator topology convergence
$$  (I-P) M^\ast P f = \lim_n ( M^\ast P_{n+1} - P_n M^\ast) f$$
holds true.

\begin{proposition}
The integral kernel $L(z,w)$ representing the Hankel operator $(I-P)M^\ast P$ admits the approximation:
\begin{equation}
\lim_n [ \overline{z} K^\mu_{n+1}(z,w) - \overline{w} K^\mu_n(z,w)] = L(z,w)
\end{equation}
in the weak topology of $L^2(\mu \times \mu)$. 
\end{proposition} 

This observation is consistent with
    $$ \lim_n  \frac{\partial M^\ast P_{n+1} f}{\partial \overline{z}} = \frac{\partial (I-P)M^\ast P f}{\partial \overline{z}} = Pf$$
in the even weaker sense of distributions.

\subsection{Polynomial approximation}

We specialize below the finite rank approximation of the commutator $[M^\ast, P]$ to the filtration by polynomial
subspaces $\C_n[z]$ labelled by the degree. The notation introduced in the previous sections is unchanged: $P_n$ is the orthogonal
projection onto $\C_n[z]$, Christoffel-Darboux kernel of order $n$ is $K_n(z,w)$, Hessenberg's matrix associated to the orthogonal polynomials
$(p_j)$ is $[h_{jk}]$.

The rate of convergence in Hilbert-Schmidt norm of $[M^\ast,P]P_n$ to $[M^\ast, P]$ controls via Theorem \ref{gap} the finite central convergence in the trace formula.
For a fixed degree $j \geq 0$ one finds
$$ \| [M^\ast, P] p_j \|^2 = \| (I-P) M^\ast p_j \|^2 = \| \overline{z} p_j \|^2 - \| P ( \overline{z} p_j ) \|^2 = \| z p_j \|^2 - \| P ( \overline{z} p_j ) \|^2 = $$ $$
\sum_{k \leq j+1} |h_{kj}|^2 - \sum_{k \geq j-1} |h_{jk}|^2.$$
That is, the square norm of the $j$-th column in the Hessenberg matrix minus the square norm of the $j$-th row.
Therefore
\begin{equation}\label{first n}
\| [M^\ast,P]P_n \|^2_{HS} =  \sum_{j \leq n} \| [M^\ast, P] p_j \|^2 = |h_{n+1,n}|^2 - \sum_{j \leq n<k} |h_{jk}|^2.
\end{equation}
We know that as a function of $n$ this is an increasing sequence, converging to ${\rm Tr} [S^\ast, S] = \frac{ {\rm Area} \Sigma(\mu)}{\pi}$.

At the level of integral formulas one puts in motion the Christoffel-Darboux kernel. Denoting by $L_n(z,w)$ the integral kernel of  $[M^\ast,P]P_n$ 
one finds
$$ L_n(z,w) = \int L(z,\zeta) K_n(\zeta,w) d\mu(\zeta), \ \ n \geq 0,$$
and one step further, approximating the full projection $P$ by $P_N$ with $N \rightarrow \infty$
we infer a formula in terms only of the moments of the original measure $\mu$.

\begin{proposition} The integral kernel $L_n$ of the finite rank operator  $[M^\ast,P]P_n$ 
admits the representation
\begin{equation}
L_n(z,w) = \lim_N L_{N,n}(z,w),$$
$$ L_{N,n}(z,w) : = \int K_N(z,\zeta) (\overline{z} - \overline{\zeta}) K_n(\zeta, w) d\mu(\zeta)
\end{equation}
with the error estimate
\begin{equation}\label{far-corner}
\| L_n(z,w) - L_{N,n}(z,w) \|^2_{2, \mu \otimes \mu} = 
$$ $$ \| (P_N - P) M^\ast P_n \|^2_{HS} = \sum_{j \leq n<N < k} |h_{jk}|^2.
\end{equation}
\end{proposition}

Given a polynomial $R(z,\overline{z}) \in \C[z, \overline{z}]$ it will be convenient to separate the powers of $\overline{z}$:
$$ R(z,\overline{z}) = R_0{z} + R_1(z) \overline{z} + \ldots + R_d(z) \overline{z}^d. $$
The identity
$$ [ R(S,S^\ast), S] = [S^\ast, S] R_1(S) + [S^{\ast 2}, S] R_2(S) + \ldots + [S^{ast d}, S] R_d(S)$$
implies, thanks to the relation $PMP = MP$ and the ciclicity invariance of trace:
$$ {\rm Tr}  [ R(S,S^\ast), S] =  {\rm Tr} \ P M (I-P) M^\ast R_1(M) + \ldots  + {\rm Tr} \ P M (I-P) M^{\ast  d} R_d(M).$$
Denoting
$$ \tilde{R}(\overline{z}; w, \overline{w}) = \sum_{j=1}^d \frac{\overline{z}^j - \overline{w}^j}{\overline{z} - \overline{w}} R_j(w),$$
formula \ref{integral} becomes:
$$ {\rm Tr}  [ R(S,S^\ast), S] = \int \int |L(z,w)|^2  \tilde{R}(\overline{z}; w, \overline{w}) d\mu(z) d\mu(w).$$
Notice that $\tilde{R}$ is a polynomial in all variables and its restriction to the diagonal $z=w$ coincides with $\frac{\partial R}{\partial \overline{z}} (z,\overline{z}).$

The main result follows. We denote by $\hat{\sigma}$ the polynomial convex hull of a closed set $\sigma \subset \C$.

\begin{theorem}\label{main} Let $\mu$ be a positive Borel measure with compact support $\sigma$ in $\C$, with associated Hessenberg matrix $(h_{jk})_{j,k=0}^\infty$ and Christoffel-Darboux kernels
$K_n(z,w), n \geq 0.$ The moments of the area measure supported by the cloud of $\mu$ can be computed by the formula:
\begin{equation}\label{Nn}
\frac{1}{\pi} \int_{\Sigma(\mu)} \frac{\partial R}{\partial \overline{z}} (z,\overline{z}) dA(z) = $$ $$
\lim_n \lim _N  \int  R(z,\overline{z}) [z K_n (z,z) - \int K_n (z,\zeta)\zeta K_N(\zeta, z) d\mu(\zeta)] d\mu(z)
\end{equation}
where $R \in \C[z,\overline{z}]$.

For fixed values of $n < N$ the error $\epsilon_{N,n} = \epsilon_{N,n}(\mu)$ in the above limit satisfies:
$$ \epsilon_{N,n}^2 \leq  \frac{ {\rm Area}\  \hat{\sigma}}{\pi} \|\tilde{R}\|^2_{\infty, \sigma \times \sigma} [(\sum_{j >n} s_j) + \sum_{j \leq n<N < k} |h_{jk}|^2],$$
where
$$ s_j = h_{j+1,j}^2 - h_{j,j-1}^2 + \sum_{\ell < j < k} (|h_{\ell j}|^2- |h_{jk}|^2)$$
and   $ \pi \sum_{j=0}^\infty s_j = {\rm Area}\  \Sigma(\mu)$.
\end{theorem}

\begin{proof} By relaxing the two projectors $P$ appearing in the trace formula
$$ \frac{1}{\pi} \int_{\Sigma(\mu)} \frac{\partial R}{\partial \overline{z}} (z,\overline{z}) dA(z) = {\rm Tr} [P M_R P, P M P] = $$ $$ {\rm Tr} P M (I-P) M_R P = {\rm Tr} P M (I-P) M_R ,$$
one finds the approximate values
$$  {\rm Tr} P_n M (I-P_N) M_R = {\rm Tr} R P_n M (I-P_N),$$
which provide the kernels (\ref{main}) in the statement. The error bound is derived from (\ref{estimate}) and the Hilbert-Schmidt norm identity
$$ \| (I-P_N) M^\ast P_n - (I-P) M^\ast P \|_{HS}^2 = $$ $$ \| (P-P_N) M^\ast P_n \|^2_{HS} + \| (I-P) M^\ast (P-P_n) \|^2_{HS}.$$
The expressions in terms of Hessenberg matrix entries are consequences of identities (\ref{first n}) and (\ref{far-corner}).

\end{proof}

Helton and Howe trace formula (\ref{HH}) is coordinate free, and moreover, it is invariant under additive Hilbert Schmidt perturbations $S+L$ of the operator $S$, provided
the self-commutator $[S^\ast + L^\ast, S+L]$ remains trace-class \cite{Voiculescu-1980'}. In particular, an adapted sequence of finite rank orthogonal projections converging strongly to the identity
operator may well improve the estimates in Theorem \ref{main}. We discuss an example in this direction.

\begin{corollary}\label{perturbation} Let $\mu$ be a positive measure of compact support and let $\nu$ be a finite atomic measure supported by the complement of the polynomial hull of ${\rm supp} \mu$.
There exists a constant $\rho>1$ with the property
\begin{equation}\label{ext-points}
 \epsilon^2_{N,n}(\mu + \nu) \leq  \epsilon^2_{N,n}(\mu) + O(\rho^{-n}), \ n \rightarrow \infty,
\end{equation}
with the second term independent of $N >n$.
\end{corollary}

\begin{proof} By finite recurrence and normalization we can assume $\nu$ is a Dirac mass at a point $a$.

The map
$$ J: P^2(\mu + \nu) \longrightarrow P^2(\mu) \oplus P^2(\nu), \ \ f \mapsto (f,f),$$
is isometric and has dense range by Runge approximation theorem. Hence $J$ is a unitary transformation.
Denote by $\pi_n$ the orthogonal projection of $P^2(\mu + \nu)$ onto $\C_n[z]$.

The essential Hilbert-Schmidt norm involved in the bound of $ \epsilon_{N,n}(\mu + \nu)$ is
$$ \| \begin{pmatrix}
        I - P & 0\\
        0   & 1-1 \end{pmatrix} \begin{pmatrix}
                                             M^\ast & 0\\
                                              0 &  \overline{a}\end{pmatrix}    \pi_n \|_{HS}.$$
                                                                                               Fix a positive integer $n$, and consider the orthogonal decomposition of $J \C_n[z]$:
$$ (K^\mu_n(z,a), K^\mu_n(a,a)) \C \oplus J ( (z-a)\C_{n-1}[z]).$$
The normalized vector
$$ (\frac{K^\mu_n(z,a)}{\sqrt{K^\mu_n(a,a) + K^\mu_n(a,a)^2}},  \frac{K^\mu_n(a,a)}{\sqrt{K^\mu_n(a,a) + K^\mu_n(a,a)^2}})$$
and an orthonormal basis of the subspace $(z-a)\C_{n-1}[z] \subset P^2(\mu)$, simultaneously orthonormal via the unitary map $J$ in $P^2(\mu) \oplus P^2(\nu)$,
produce the estimate
$$ \|  \begin{pmatrix}
        (I - P)M^\ast & 0\\
        0   & 0 \end{pmatrix}  \pi_n \|^2_{HS} \leq $$ $$
                                               \| (I-P) M^\ast \frac{K^\mu_n(z,a)}{\sqrt{K^\mu_n(a,a) + K^\mu_n(a,a)^2}}\|^2 + \| (I-P)M^\ast P_n\|^2.$$
Since the point $a$ does not belong to the polynomial convex hull of $ {\rm supp}(\mu)$  the variational definition of Christoffel function
$ \Lambda^\mu_n (a) $ implies 
$$ \Lambda^\mu_n (a) \leq C \rho^{-n},$$
for some constants $C>0$ and $\rho >1$. The identity
$$\| \frac{K^\mu_n(z,a)}{\sqrt{K^\mu_n(a,a) + K^\mu_n(a,a)^2}} \|^2_\mu  = \frac{K^\mu_n(a,a)}{K^\mu_n(a,a) + K^\mu_n(a,a)^2} =
\frac{\Lambda^\mu_n(a)}{1+ \Lambda^\mu_n(a)}$$    
completes the proof.                                       
\end{proof}

\section{Pad\'e type approximation scheme in 2D}

The reconstruction of the cloud of a measure $\mu$ from its moments can be completed via different paths: geometric tomography, Bergman space methods
\cite{GPSS-2009} (if applicable), identification of potential real-algebraic boundary \cite{Lasserre-Putinar-2015}, curve fitting along the boundary \cite{Ammari-2019}, or by exploiting the same hyponormal operators tools invoked in the previous sections. We reproduce from our preceding article a few details on the latter approximation scheme.

Let $g \in L^1_{\rm comp}(\C, dA), \ \ 0 \leq g \leq 1,$ be a fixed shade function, and think of it as the characteristic function of the shade $\Sigma(\mu)$ of a measure.
We consider the moments 
$$ 
a_{k\ell} = \int_\C \zeta^k \overline{\zeta}^\ell g(\zeta) {\rm dA}(\zeta), \ \ k,l \geq 0,
$$
given. They can be organized in the exponential of a formal generating series:
$$ E_g(w,z) = \exp [\frac{-1}{\pi} \sum_{k,\ell=0}^\infty \frac{a_{k\ell}}{w^{k+1} \overline{z}^{\ell+1}}].$$
This is of course the power expansion at infinity of the double Cauchy integral appearing in (\ref{det}).
Denote in short:
$$ E_g(w,z) =  \exp ( \frac{-1}{\pi} \int_\C \frac{ g(\zeta) {\rm dA}(\zeta)}{(\zeta-w)(\overline{\zeta}-\overline{z})}), \ \ z,w \in \C, \ z\neq w.$$

We recall a few of the properties of the {\it exponential transform} $E_g$:
\bigskip

a). The function $E_g$ can be extended by continuity to $\C^2$ by assuming the value $E_g(z,z) = 0$ whenever $ \int_\C \frac{ g(\zeta) {\rm dA}(\zeta)}{|\zeta-z|^2} = \infty;$

b). The function $E_g(w,z)$ is analytic in $w \in \C \setminus {\rm supp} (g)$ and antianalytic in $z \in \C \setminus {\rm supp} (g)$;

c). The kernel $1-E_g(w,z)$ is positive semi-definite in $\C^2$;

d). The behavior at infinity contains as a first term the Cauchy transform of $g$:
$$ E_g(w,z) = \frac{1}{\overline{z}} [\frac{-1}{\pi} \int_\C \frac{ g(\zeta) {\rm dA}(\zeta)}{\zeta-w}] + O(\frac{1}{|z|^2}), \ \ |z| \rightarrow \infty.
$$ 
\bigskip

The case of a characteristic function $g= \chi_\Omega$ of a bounded domain $\Omega$ is particularly relevant for our note. In this case we simply write
$E_\Omega$ instead of $E_g$, and we record the following properties (all proved and well commented in \cite{Gustafsson-Putinar-2017}).
\bigskip

1). The equation 
$$ \frac{\partial E_\Omega(w,z)}{\partial \overline{w}} = \frac{E_\Omega(w,z)}{\overline{w}-\overline{z}}$$
holds for $z \in \Omega$ and $w \in \C \setminus \Omega$;

2). The function $E_\Omega(w,z)$ extends analytically/antianalytically from \\
$(\C \setminus \overline{\Omega})^2$ across real analytic arcs of the boundary of  $\Omega$;

3).  Assume $\partial \Omega$ is piecewise smooth. Then $z \mapsto E_\Omega(z,z)$ is a superharmonic function on the complement of $\Omega$,
with value $1$ at infinity, vanishing on $\overline{\Omega}$ and satisfying
$$ E_\Omega(z,z) \approx {\rm dist} (z, \partial \Omega)$$
for $z \in \C \setminus \overline{\Omega}$ close to $\partial \Omega$.
\bigskip

The kernel $E_\Omega$ is also characterized by a Riemann-Hilbert factorization, see \cite{Gustafsson-Putinar-2017}. The feature which turns the exponential transform into a suitable shape reconstruction from moments tool is its rationality  on quadrature domains, in complete parallelism to the rationality of Cauchy transforms of point masses on the line, or the rationality of exponential of Cauchy transforms 
of union of intervals. 

We reproduce the rational reconstruction procedure. Let $d$ be a fixed integer and let
$(a_{k\ell})_{k,\ell=0}^d,$
be a non-negative matrix of potential moments of a ``shade function" $g(z), 0 \leq g \leq 1$. Consider the truncated exponential transform
$$ F(w,z) =  \exp [\frac{-1}{\pi} \sum_{k,\ell=0}^d \frac{a_{k\ell}}{w^{k+1} \overline{z}^{\ell+1}}] = $$ $$1 - \sum_{m,n=0}^\infty \frac{b_{mn}}{w^{m+1} \overline{z}^{n+1}}.$$
It is known that $(b_{mn})_{M,n=0}^d$ is also a non-negative definite matrix, subject to some additional constraints \cite{Martin-Putinar-1989}.
A necessary and sufficient condition that $(a_{k\ell})_{k,\ell=0}^d$ represent the moments of a quadrature domain of order $d$ is
$$ \det (b_{mn})_{m,n=0}^d =0,$$
or equivalently the existence of a monic polynomial $P(z)$ of degree $d$ and a rational function of the form
$$ R_d(w,z) = 1- \frac{\sum_{m,n=0}^{d-1} c_{mn} w^m \overline{z}^n}{P(w) \overline{P(z)}},$$
such that, at infinity
$$ F(w,z) - R_d(w,z) = O( \frac{1}{w^{d+1} \overline{z}^d},  \frac{1}{w^{d} \overline{z}^{d+1}}).$$
The reader will recognize above a typical 2D Pad\'e approximation scheme. Moreover, for any shade function $g$, the exponential transform
$E_g$ coincides with $E_\Omega$, where $\Omega$ is a quadrature domain if and only if 
$$E_g(w,z) = 1- \frac{\sum_{m,n=0}^{d-1} c_{mn} w^m \overline{z}^n}{P(w) \overline{P(z)}}, \ \ |z|, |w| \gg 1.$$
In this case the zeros of $P$ coincide with the quadrature nodes, while the numerator is the irreducible defining polynomial of the boundary of $\Omega$:
$$ \partial \Omega \subset \{ z \in \C; \ \sum_{m,n=0}^{d-1} c_{mn} z^m \overline{z}^n = |P(z)|^2 \}.$$

The above Pad\'e approximation procedure was proposed for the reconstruction of planar shapes in \cite{GHMP-2000}, with additional details in \cite{Gustafsson-Putinar-2017}.

\section{Examples}

\subsection{Reconstruction of a disk via its exponential transform}

As simple and well known the example below might be, it is illustrative for the two dimensional Pad\'e scheme just discussed.

A disk $B = \{ z \in \C; |z-c| <R\}$ has the exponential transform
$$E_B(z,w) = 1 - \frac{R^2}{(z-c)(\overline{z}-\overline{c})}$$
detectable from initial moments:
$$ a_{00} = \pi R^2,$$
$$ a_{01} = \int_{|z-c| \leq R} z dA(z) = \pi R c = \overline{a_{10}},$$
$$ a_{11} = \int_{|z-c| \leq R} |z|^2 dA(z) = 2\pi \int_0^R (|c|^2 + r^2) r dr = \pi R^2 |c|^2 + \pi \frac{R^4}{2}.$$
The truncated exponential transforms is:
$$ \exp [- \frac{R^2}{z \overline{w} } - \frac{R^2 \overline{c}}{z\overline{w}^2} - \frac{R^2 c}{z^2 \overline{w}} - \frac{R^2 |c|^2 + \frac{R^4}{2}}{z^2 \overline{w}^2}]  = $$ $$
 1- \frac{R^2}{z \overline{w} } - \frac{R^2 \overline{c}}{z\overline{w}^2} - \frac{R^2 c}{z^2 \overline{w}} - \frac{R^2 |c|^2}{z^2 \overline{w}^2} +
 O( \frac{1}{w^{3}},  \frac{1}{\overline{z}^{3}}).$$
 Whence
 $$ b_{00} = R^2,  \ b_{10} = R^2 c,\ b_{01} = R^2 \overline{c}, \ b_{11} = R^2 |c|^2.$$
 The vanishing determinant $b_{00}b_{11}-b_{10}b_{01} = 0$ and the linear dependence of the columns of the matrix $(b_{k\ell})_{k,\ell =0}^1$
 identify the monic factor $P(z) = z-c$ in the denominator $P(z)\overline{P(w)}$ of the rational approximant of the full exponential transform. Finally, as in the one dimensional diagonal Pad\'e approximation scheme, one finds:
 $$ (z-c)(\overline{w}-\overline{c})  [1- \frac{R^2}{z \overline{w} } - \frac{R^2 \overline{c}}{z\overline{w}^2} - \frac{R^2 c}{z^2 \overline{w}} - \frac{R^2 |c|^2}{z^2 \overline{w}^2} ] = $$ $$
 (z-c)(\overline{w}-\overline{c})  - R^2 + O(\frac{1}{z^2}, \frac{1}{\overline{w}^{2}}).$$

\subsection{Rotationally invariant measures}
Let $\rho$ denote a positive Borel measure on the interval $[0,1]$ with $1 \in {\rm supp}(\rho)$. That is
$\rho ([\delta, 1]) >0$ for every $ \delta < 1$. The induced rotationally invariant measure $\mu$ acts on smooth functions
$\phi(x,y) $ by the formula
$$ \int \phi d\mu = \int \phi (r \cos \theta, r \sin \theta) d\rho (r).$$
The support of the measure $\mu$ contains full circles, hence it is not finite.
The complex monomials are orthogonal in $L^2(\mu)$, with norms
$$ \frac{1}{\gamma_k^2} = \| z^k \|^2_{2,\mu} = \int r^{2 k} d\rho(r), \ \ k \geq 0.$$
The orthonormal polynomials are
$$ p_k(z) = \gamma_k z^k, \ \ k \geq 0,$$
so that the multiplier $S = M_z$ acts as a weighted shift:
$$ S p_k = h_{k+1,k} p_{k+1}, k \geq 0.$$
Note that  $(\frac{1}{\gamma_k^2})_{k=0}^\infty$ are the moments of a positive measure defined on $[0,1]$
(specifically $d\rho(\sqrt{r})$), and
$$  h_{k+1,k} = \frac{\gamma_{k}}{\gamma_{k+1}} > 0, \ \ k \geq 0.$$
All other entries in the associated Hessenberg matrix $H$ are equal to zero. In short,  $H$ is a subnormal weighted shift.

The spectrum of $H$ is well understood: it coincides with the closed unit disk. Moreover, the elements of $P^2(\mu)$ are
analytic functions in the open disk, subject to a growth condition imposed by the coefficients $\gamma_k$, cf. \cite{Shields-1974}.
In particular, $P^2(\mu) \neq L^2(\mu)$ and every  $\lambda \in \D$ is a bounded point evaluation for $P^2(\mu)$. That is, the cloud
of $\mu$ is the full disk: $\Sigma(\mu) = \overline{\D}$. 

We give for completeness a few details. Condition $[S^\ast, S] \geq 0$ reads
$$ \| S p_k \|^2 \geq \|S^\ast p_k \|^2,$$
or equivalently
$$ h^2_{k+1,k} \geq h^2_{k, k-1}, \ \ k \geq 1.$$
In addition
$$ \lim_k h_{k+1,k} =1.$$
Indeed, 
$$ S^\ast S p_k =  h^2_{k+1,k} p_k, \ \ k \geq 0,$$
and $ \| S^\ast S \| = \| S \|^2 = \| M \|^2 = 1.$

According to Theorem \ref{main}, the rate of convergence of the approximation scheme is dictated by the remainder:
$$  [(\sum_{j >n} s_j) + \sum_{j \leq n<N < k} |h_{jk}|^2] = 1- h_{n+1,n}^2.$$
In general, for Hausdorff moment sequences such as  $(\frac{1}{\gamma_k^2})_{k=0}^\infty$ above, the convergence rate of
consecutive quotients is known:
$$ h^2_{k+1,k} = \frac{\gamma^2_{k}}{\gamma^2_{k+1}} = 1 - O(\frac{1}{k})$$
A multi-fractal gauge, known as the local dimension of a measure, quantifies this asymptotics,
cf. \cite{Grabner-Prodinger-1996,Olsen-2016}.

The remarkable feature of this class of examples is that all (normalized) rotationally invariant measures share the same cloud.

\subsection{A uniform mass cloud plus finitely many point masses} Let $\Omega \subset \C$ be a bounded, connected and simply connected domain
with smooth boundary. The uniform area mass $\nu = \chi_\Omega dA$ distributed on $\Omega$ offers one of the best understood asymptotic analysis of complex orthogonal polynomials,
with a century old history, see for instance \cite{Suetin-1974}.
In this case $P^2(\nu)$ coincides with the {\it Bergman space} $L^2_a(\Omega)$, that is the collection of all analytic functions in $\Omega$, square summable with respect to area.The reproducing kernel
$$K^\Omega(z,w) = \sum_{j=0}^\infty p_j(z) \overline{p_j(w)} $$
converges in $\Omega \times \Omega$ to an analytic/anti-analytic positive definite kernel, known as the {\it Bergman kernel} of $\Omega$.
If $\phi : \Omega \longrightarrow \D$ denotes a conformal mapping, then
$$ K^\Omega(z,w) = \frac{\phi'(z) \overline{\phi'(w)}}{\pi (1- \phi(z) \overline{\phi(w))^2}}, \ \ z, w \in \Omega.$$
Therefore, the integral kernel representing $(I-P)M^\ast P$ is precisely
$$ L(z,w) = (\overline{z} - \overline{w}) K^\Omega(z,w).$$
The asymptotics of the singular numbers of this integral operator (known as a big Hankel operator) were thoroughly studied \cite{Arazy-1998,Janson-1988}.
In general, one carries to the unit disk all computations, via the inverse conformal mapping $\phi = \psi^{-1}$. The integral kernel
$$ L_1(u,v) = \frac{ \overline{\phi(u)}- \overline{\phi(v)}}{\pi (1- u \overline{v})^2}, \ \ u,v \in \D$$
gives rise to a unitarily equivalent integral operator to $(I-P)M^\ast P$, this time acting on $L^2(\D, dA)$. Within this framework, a theorem due to Nowak \cite{Nowak-1991} asserts that the singular numbers $\kappa_j$
of the Hankel operator $(I-P)M^\ast P$ associated to a domain $\Omega$ with {\it smooth} boundary satisfy
$ \kappa_j = O(\frac{1}{j})$.
Consequently the best approximation by a sequence of finite rank projections $Q_n, \ {\rm rank} \ Q_n = n,$ yields the error $\kappa_{n+1}^2 + \kappa_{n+2}^2 + \ldots$, that is
$$ \| (I-P) M^\ast (P-Q_n)\|_{HS}^2 = O(\frac{1}{n}).$$

We show that this estimate is sharp for an ellipse, with respect to the polynomial filtration. More precisely, let $E$ be the ellipse whose complement
is described by Joukowski's map $$F(z) = \frac{1}{2}( e^c z + \frac{1}{e^c z}), \ |z| > 1,$$ and parameter $c >0$.
The associated Hessenberg matrix is three diagonal, due to the special structure of the orthogonal polynomials with respect to area measure on $E$. Indeed,
$$ p_j(z) = 2 \sqrt{\frac{j+1}{\pi}} \frac{1}{\sqrt{\rho^{j+1} - \rho^{-j-1}}} U_j(z), \ \ j \geq 0,$$
where $\rho = e^{2c}$ and $U_j$ denotes Chebyshev polynomial of the second kind, see for details \cite{Nehari-1975} pg. 259. We adopt the convention $p_{-1} =0$ as well $U_{-1} =0$.
The three term recurrence relation for Chebyshev polynomials
$$ 2 z U_j(z) = U_{j+1}(z) + U_{j-1}(z), \ j \geq 0,$$
implies
$$ z p_j(z) = \frac{1}{2} \sqrt{\frac{j+1}{j+2}} \sqrt{ \frac{ \rho^{j+2}-\rho^{-j-2}}{\rho^{j+1} - \rho^{-j-1}}} p_{j+1} +  \frac{1}{2} \sqrt{\frac{j+1}{j}} \sqrt{ \frac{ \rho^{j}-\rho^{-j}}{\rho^{j+1} - \rho^{-j-1}}} p_{j-1}.$$
In other terms, $h_{jj}= 0$,
$$ h_{j+1,j} = \frac{1}{2} \sqrt{\frac{j+1}{j+2}} \sqrt{ \frac{ \rho^{j+2}-\rho^{-j-2}}{\rho^{j+1} - \rho^{-j-1}}} $$
and
$$ h_{j-1,j} = \frac{1}{2} \sqrt{\frac{j+1}{j}} \sqrt{ \frac{ \rho^{j}-\rho^{-j}}{\rho^{j+1} - \rho^{-j-1}}}.$$

In order to evaluate the approximation rate in Theorem (\ref{main}) we remark that  $\sum_{j \leq n<N < k} |h_{jk}|^2 =0$
whenever $N>n$, hence 
$$  \frac{1}{\pi} {\rm Area \ E} - (h_{n+1,n}^2 -h_{n,n+1}^2) $$
dictates, up to the stated constants, the rate of convergence. Since ${\rm Area \ E} = \frac{1}{4} (\rho - \rho^{-1})$ with $\rho >1$, we find the error estimate:
$$ \frac{1}{4} (\rho - \frac{1}{\rho}) - \frac{1}{4} [ \frac{n+1}{n+2} \rho \frac{1-\rho^{-2n-3}}{1-\rho^{-2n -2}} - \frac{n+2}{n+1} \frac{1}{\rho} \frac{1-\rho^{-2n-2}}{1-\rho^{-2n -3}}] = O(\frac{1}{n}).$$

Conformal and quasi-conformal mapping techniques led recently to sharp estimates of the Hessenberg matrix entries of the subnormal multiplier $S_\nu = M_z$, acting on $P^2(\nu)$, \cite{BS-2018}.
This article represents the highest point of several decades of accumulated studies, by many authors. We mention the main setting.

Let 
$$F(z) = c_{-1} z + c_0 + c_1 \frac{1}{z} + c_2 \frac{1}{z^2} + \ldots $$
denote the conformal mapping of the exterior of the closed unit disk onto $\C \setminus \overline{\Omega}$, It is customary to normalize $F$ by the condition $c_{-1} >0$. Theorem 1.2 of
\cite{BS-2018} asserts that there exists a constant $\beta \geq 1$ so that Hessenberg matrix of $S_\nu$ is asymptotically close to the Toepliz matrix $\mathcal T$ with entries $c_{-1}, c_0, c_1, \ldots $. More specifically:
$$ | h_{n-k,k} - \sqrt{\frac{n+1}{n-k+1}} c_k| = O( \frac{1}{n^\beta}), \ \ n \rightarrow \infty,$$
where $O$ depends on $k \geq -1$.
In particular, the only non-zero under-diagonal terms satisfy:
$$ |h_{n+1,n}  - \sqrt{\frac{n+1}{n+2}} c_{-1}| = O( \frac{1}{n^\beta}), \ \ n \rightarrow \infty.$$
That is
$$ |h_{n+1,n} - c_{-1}| = O(\frac{1}{n}).$$
The value of the constant $\beta$ depends on the regularity of the boundary
of $\Omega$, \cite{BS-2018}.

Remark that the series $\sum_{\ell \geq 1} \ell |c_\ell|^2$ converges via a well known area estimate, in its turn a consequence of Stokes formula: 
$$ \frac{1}{\pi}  \ {\rm Area \Omega} =  |c_{-1}|^2 - |c_1|^2 - 2 |c_2|^2 - 3 |c_3|^2 - \ldots.$$
Just for validation: this is nothing else than the trace of the self-commutator of the corresponding Toeplitz matrix $\mathcal T$. 

Recall from Theorem \ref{main} that the orthogonal projection $P_n$ onto $\C_n[z]$ satisfies the identity
$$ \| (I-P)M^\ast P_n \|^2_{HS} = h_{n+1,n}^2 - \sum_{j \leq n< k} |h_{jk}|^2.$$
 
In conclusion,
$$ \lim_n  \sum_{j \leq n< k} |h_{jk}|^2  = |c_1|^2 + 2 |c_2|^2 + 3 |c_3|^2 + \ldots .$$
The yet unknown rate of convergence in the latter limit is not expected to be better than $O(\frac{1}{n})$, as the ellipse case shows.

One of the first study of estimates of orthogonal polynomials in Bergman space setting is due to Carleman, in the case of real analytic boundaries, see for instance \cite{Suetin-1974}. Without entering into details, we mention that in this scenario there exists $\rho>1$, depending on the geometry of $\partial \Omega$ (how far Schwarz function of this curve analytically extends inside $\Omega$), such that
the decay in formula (\ref{far-corner}) is geometric:
$$ \sum_{j \leq n < N < k} |h_{jk}|^2 = O(\frac{1}{\rho^N}),$$
where $O$ depends on $n$.

In general, Corollary \ref{perturbation} allows to estimate the error term in the moment approximation formula (\ref{main}) even after adding finitely point masses outside the polynomial convex hull of $\overline{\Omega}$.

\subsection{Finite rank self-commutator} Even the simplest, finite rank self-commutator scenario raises challenging approximation theory questions. 
In view of the structural theorem of McCarthy and Yang, we have to focus in this case on a rational conformal map $r : \D \longrightarrow \Omega$ of the disk
onto a bounded quadrature domain, and the push forward measure $\mu = r_\ast ( d\theta + \nu)$, where $d\theta$ is arc length on the unit circle and
$\nu$ is a finite atomic, positive measure supported by $\D$. Then the operator $S_\mu = PM_z P$ is a cyclic, irreducible subnormal with finite rank self-commutator, and vice-versa
 \cite{McCarthy-Yang-1995,McCarthy-Yang-1997}. Let the Schmidt expansion of $(I-P)M^\ast P$ be:
 $$  (I-P)M^\ast P= \sum_{j=0}^d \kappa_j g_j \langle \cdot, f_j \rangle,$$
 with $d$ finite. Every eigenfunction $f_j, \ 0 \leq j \leq d,$ of the self-commutator $[S^\ast,S]$ annihilates a finite codimension ideal of the ring of analytic functions defined on $\Omega$
 (see \cite{McCarthy-Yang-1995}). Let $a_1, \ldots,a_p \in \Omega$ denote the support of this ideal. Therefore every $f_j$ is a linear combination of the corresponding point evaluation
 functionals $k_{a_1}, \ldots, k_{a_p}$. Due to the definition of the measure $\mu$, these point evaluation functionals are push forward via $r$ of point evaluations functionals with respect to the measure 
 $d\theta + \nu$. And the latter can be explicitly computed recursively (via the so-called Uvarov's transform). We indicate only one step of this transform, corresponding to the measure $d\theta + \delta_\alpha$
 with $\alpha \in D$:
 $$ K^{d\theta + \delta_\alpha}(z,w) = \frac{1}{2 \pi} [ \frac{1}{1-z\overline{w}} + \frac{C}{(1-z \overline{\alpha})(1-\alpha \overline{w})}], \ \ z,w \in \Omega,$$
 where $C$ is a constant.
More details can be found in \cite{Simon-OP2}.
We infer from these formulas that every evaluation functional $K^{d\theta + \nu}(z,w)$ analytically extends as a function of $z$, across $\partial \D$, for a fixed value of $w \in \D$.
The same analytic continuation feature carries to $k_{a_1}, \ldots, k_{a_p}$ provided the boundary of $\Omega$ is smooth. Recall that the only singular points in the boundary of a quadrature domain
are inner cusps \cite{Gustafsson-Shapiro-2005}. We discuss a generic situation.

\begin{proposition} Let $S$ be a cyclic subnormal operator with finite-rank self-commutator, so that its spectrum is the closure of a finite union $\Omega$ of quadrature domains, plus a finite number of points.
If the boundary of $\Omega$ is smooth, then there exists $\rho>1$, so that the error in Theorem \ref{main} satisfies:
$$ \| (I-P)M^\ast (P-P_n) \|_{HS} = O(\frac{1}{\rho^n}), \ n \rightarrow \infty.$$
\end{proposition}

\begin{proof} The non-degenerate case assumed in the statement, implies that $\Omega$ is a real analytic curve without singularities.
Let $f_1, \ldots,f_p$ denote the eigenfunctions of $[S^\ast, S]$. 
We just proved that $f_1, \ldots,f_p$ are 
analytic functions defined in a neighborhood of $\overline{\Omega}$. 

The full Hilbert-Schmidt norm of the main Hankel operator is
$$ \| (I-P)M^\ast P \|^2_{HS} = \sum_{j=0}^p \| (I-P)M^\ast f_j \|^2,$$
while the error of interest is
$$ \| (I-P)M^\ast (P-P_n) \|^2_{HS} = \sum_{j=0}^p \| (I-P)M^\ast (f_j - P_n f_j)\|^2.$$
It remains to prove that there exists a constant $\rho >1$, so that, for every $j, 0 \leq j \leq p$:
$$ \inf \{ \| f_j - h \|_{2,\mu}, \  h \in \C_n[z] \} = O ( \frac{1}{\rho^n}), \ n \rightarrow \infty.$$
Since the complement of $\overline{\Omega}$ is connected, a Theorem of Russell and Walsh
insures the above decay, even with respect to uniform norm on $\overline{\Omega}$, see \cite{Walsh-1965} Section 4.7.

\end{proof}

The case of non-smooth boundary is different, and more intriguing. We consider a simple example.
Let $r(z) = (z-1)^2$ be the conformal map of the disk onto the cardiodid $\Omega$, a quadrature domain of order two.
The point $0 = r(1)$ is a singular point of $\Omega$, where the boundary has an inner cusp.

Let $U = M_z$ denote the unilateral shift, on Hardy space $H^2(\D) = P^2(\T, d\theta)$. The monomials $1, z, z^2 , \ldots$ form an orthonormal basis, with respect to normalized 
arc length measure $\frac{d \theta}{2\pi}$. The operator $S_\mu$ corresponding to the push forward measure $r_\ast \frac{d \theta}{2\pi}$ is unitarily equivalent to $(U-I)^2$.
The action of $U$ on the basis is by shifting the indices $U z^n = z^{n+1}, n \geq 0,$ with $U^\ast z^{n+1} = z^n$ and $U {\mathbf 1} = 0$. One computes immediately the self commutator,
in this Hilbert space representation:
$$ [S^\ast_\mu, S_\mu] = [ I - 2 U^\ast + U^{\ast 2}, I - 2 U + U^2] = $$ $$ 4 [U^\ast, U] - 2 [U^\ast, U^2] - 2[U^{\ast 2}, U] + [U^{\ast 2}, U^2]$$
remarking that every commutator in the last expression annihilates $z^k, \ k \geq 2$. Therefore the two eigenfunctions $f_0, f_1$ of $[S^\ast_\mu, S_\mu]$
are linear combinations of ${\mathbf 1}$ and  $z$. In view of the preceding proof, the rate of convergence in Theorem \ref{main} is a factor of
$$ \delta_n = \inf \{ \| z - h((1-z)^2)\|_{2, d\theta}, \ h \in \C_n[u] \},$$
clearly equivalent to the best polynomial approximation (in the corresponding Lebesgue space norm) of the function $\sqrt{w}$ on the boundary of $\Omega$.
Remark that the function $\sqrt{w}$ is analytic and continuous on the closure $\overline{\Omega}$ of the cardiodid, but it is not analytic in a neighborhood of it.
The same theorem of Russell and Walsh and a refined Bernstein-Markov Inequality imply that $\delta_n$ converges to zero at a slower than geometric rate. 
We do not expand here the technical details related to the cusp singularity in the boundary, see for instance \cite{Baran-1994}. The exact asymptotic decay of $\delta_n$ remains unknown to us.
It is worth mentioning that a slight perturbation in the definition of the measure $\mu$, for instance
$$ \mu_\epsilon = (r_\epsilon)_\ast d\theta, \ \ r_\epsilon(z) = (z-1-\epsilon)^2, \ \epsilon >0$$
remains within the class of rank-two self-commutator, this time with a cloud possessing a smooth boundary.

\subsection{Non-smooth clouds} Given the full generality of our approach, a vast array of pathologies enters into the picture. We simply make aware the reader of such pitfalls on
one of the simplest examples, stressing that the cloud of a measure is only contained in its {\it closed} support.

Let $\Gamma$ be a Jordan curve in the complex plane, with positive area. As for instance constructed by Osgood \cite{Osgood-1903}. Let $\Omega$ be the interior of $\Gamma$.
A celebrated theorem of Carleman, see for instance \cite{Suetin-1974}, asserts that complex polynomials are dense in the associated Bergman space. In our notation
$L^2_a(\Omega) = P^2(\chi_\Omega dA)$. Denote $\mu = \chi_\Omega dA$ and $S_\mu$ the associated subnormal operator, equal to the multiplication by $z$ on this space. The set of bounded point evaluations for $\mu$ is equal to $\Omega$, also equal to the non-essential spectrum of $S_\mu$. The spectrum of $S_\mu$ is equal to $\overline{\Omega}$. We know that the principal function $g$ of $S_\mu$ is equal to the characteristic function of $\overline{\Omega}$, modulo area null-sets. Therefore the cloud $\Sigma(\mu)$ is equal to $\overline{\Omega}$, and
$$ {\rm Area} (\Sigma(\mu) \setminus \Omega) >0,$$
while
$$ \mu (\Sigma(\mu) \setminus \Omega) =0.$$

\bibliography{CDbibl}
\bibliographystyle{plain}

\end{document}